\documentclass[]{article}

\usepackage[caption=false]{subfig}

\usepackage{xcolor}
\usepackage{amsmath}        
\usepackage{amsthm}
\usepackage{amssymb}
\usepackage{mathtools}
\usepackage[numbers,sort&compress]{natbib}

\usepackage{authblk}
\theoremstyle{plain}

\usepackage{csquotes}
\usepackage{todonotes}

\newtheorem{definition}{Definition}[section]
\newtheorem{theorem}[definition]{Theorem}
\newtheorem{lemma}[definition]{Lemma}

\newtheorem{assumption}[definition]{Assumption}

\newtheorem{proposition}[definition]{Proposition}

\newcommand{\uncset}{\mathcal{U}}

\newcommand{\nominalxi}{\overline{\uncelem}}

\newcommand{\N}{\mathbb{N}}

\newcommand{\R}{\mathbb{R}}
\newcommand{\uncelem}{u}
\newcommand{\extelem}{\bar{\uncelem}}

\newcommand{\coverSpace}{\mathcal{W}} 

\newcommand{\objfunc}{f}
\newcommand{\uncobjfunc}{f_u}
\newcommand{\consfunc}{g}

\newcommand{\nameuncoptprob}{\mathcal{P}_u}
\newcommand{\closure}[1]{\overline{#1}}

\newcommand{\inverseRobustProblem}{(\mathcal{P}_\text{IROP})}
\newcommand{\GSIPApproximation}{(\mathcal{P}_{\text{GSIP}})}
\newcommand{\reducedExampleProblem}{(\mathcal{P}_{\text{red}})}

\title{A unified approach to inverse robust optimization problems}
\date{}
\author{H. Berthold, T. Heller, T. Seidel\thanks{Corresponding author. Email: tobias.seidel@itwm.fraunhofer.de}~}

\affil{Fraunhofer Institute for Industrial Mathematics ITWM\\  67663 Kaiserslautern\\  Germany}

\begin{document}

\maketitle

\begin{abstract}
A variety of approaches has been developed to deal with uncertain optimization problems. Often, they start with a given set of uncertainties and then try to minimize the influence of these uncertainties. Depending on the approach used, the corresponding price of robustness is different. The reverse view is to first set a budget for the price one is willing to pay and then find the most robust solution.

In this article, we aim to unify these inverse approaches to robustness. We provide a general problem definition and a proof of the existence of its solution. We study properties of this solution such as closedness, convexity, and boundedness. We also provide a comparison with existing robustness concepts such as the stability radius, the resilience radius, and the robust feasibility radius. We show that the general definition unifies these approaches. We conclude with examples that demonstrate the flexibility of the introduced concept.
\end{abstract}


\section{Introduction}
In many real-world problems, one does not know exactly the input data of a formulated optimization problem. This may be due to the fact that we are dealing with forecasts, predictions, or simply unavailable information. To deal with this, it is essential to treat the given data as uncertain. In principle, there are two different ways to deal with uncertainty. Either one knows some distribution of the uncertainty, or not. In the first case, this information can be used for the mathematical optimization problem, while in the second case, no additional information is given. Both approaches are widely used in many real-world applications, such as energy management, finance, scheduling, and supply chain. For a detailed overview of possible applications of robust optimization, we refer to \cite{bertsimas2011theory}. In this article, we focus mainly on problems without information about the distribution of uncertainty.

Fixing the uncertainty to solve the corresponding optimization problem may yield a solution that is infeasible for other scenarios of the uncertainty set. Therefore, one tries to find solutions that are feasible for all possible scenarios of the uncertainty set. The problem of finding an optimal solution, i.e. the solution with the best objective function value, among these feasible solutions is called the robust counterpart (cf. \cite{ben2009robust}). 

There are many surveys on robust optimization, such as Ben-Tal et al. \cite{ben2009robust} or Bertsimas et al. \cite{bertsimas2011theory}. For tractability reasons, the focus is often limited to robust linear or robust conic optimization. Robust optimization in the context of semi-infinite optimization can be found e.g. in \cite{goberna2013robust}, while \cite{lopez2007semi, vazquez2008generalized, Stein2012, stein2003solving} consider general solution methods. For applications and results on robust nonlinear optimization, we refer to a survey by Leyffer et al. \cite{leyffer2020survey}.

The question of how to construct an appropriate uncertainty set is often not addressed, and the uncertainty set is assumed to be given. A closely related question is which subset of the uncertainty set is covered by a given solution. Considering a larger uncertainty set may lead to overly conservative solutions, since more and more scenarios have to be considered. This trade-off between the probability of violation and the effect on the objective function value of the nominal problem is called the \emph{price of robustness} and was introduced by Bertsimas and Sim \cite{bertsimas2004price}. Many robust concepts that have been formulated and analyzed in recent years try to deal with the price of robustness in order to avoid or reduce it. 

Bertsimas and Sim \cite{bertsimas2003robust,bertsimas2004price} defined the \emph{Gamma robustness} approach, where the uncertainty set is reduced by cutting out less likely scenarios. The concept of \emph{light robustness} was first defined by Fischetti and Monaci \cite{fischetti2009light} and later generalized by Schöbel \cite{schobel2014generalized}. Given a tolerable loss for the optimal value of the nominal solution, one tries to minimize the \emph{grade of infeasibility} over all scenarios of the uncertainty set. 

Another approach to deal with overly conservative solutions is to allow a second stage decision. Ben-Tal et al. \cite{ben2004adjustable} introduced the idea of \emph{adjustable robustness}, where the set of variables is divided into here-and-now variables and wait-and-see variables. While the former need to be chosen before the uncertainty is revealed, the latter need to be chosen only after the realization is known.

In this article, we pursue a different approach to dealing with the price of robustness, which we call \emph{inverse robustness}. The main idea is to reverse the perspective of the approaches described above. Instead of finding a solution that minimizes (or maximizes) the objective function under a given set of uncertainties, we want to find a solution that maximizes the considered set of uncertainties under a given objective function. In this way, we are not dependent on the a priori choice of the uncertainty set and then accepting the loss of objective value. Instead, we can set the price we are willing to pay and then find the most robust solution with this given budget. Furthermore, the study of the above approaches is often limited to the robust linear case. We want to define inverse robustness in a more general way and study the concept also for nonlinear problems.

Especially for the linear case, concepts have been introduced to measure the robustness of a given solution. The \emph{stability radius} and \emph{resilience radius} of a solution can be seen as measures for a fixed solution of how much the uncertain data can deviate from a nominal value while still being an (almost) optimal solution. For a more detailed discussion of resilience we refer to \cite{Weiss2016}. Both concepts can be seen as properties of a given solution, and the shape of the uncertainty set must be specified in advance. A similar concept has been studied in the area of facility location problems. Labbé presented in \cite{labbe1991sensitivity} an approach to compute the sensitivity of a facility location problem. Several publications (\cite{carrizosa2003robust, carrizosa2015threshold,ciligot2014robustness,blanquero2011locating}) deal with the question of how to find a solution that is least sensitive, and thus deal with a concept quite similar to resilience. We will show that finding a point that maximizes the stability radius or the resilience radius, given a budget on the objective, can be seen as a special case of inverse robust optimization. However, the general definition of inverse robustness provides more flexibility. First, it allows to define measures that can include distributional information about the uncertainty. Second, the shape of the considered uncertainty is not restricted to given shapes, but can be more complex.

The outline of the article is as follows. In Section~\ref{sec: model} we define the inverse robust optimization problem (IROP) and discuss the properties of its solution. In Section~\ref{sec: cover space choice} we discuss different possible choices and description for the cover space that contains all potential uncertainty sets. Afterwards we compare our general definition with other inverse robustness concepts in Section~\ref{sec: concept comparison}. In Section~\ref{sec: examples} we provide and discuss examples. Finally, we conclude the article with a short outlook.

\section{The inverse robust optimization problem} \label{sec: model}
In this article, we consider parametric optimization problems given by
\begin{align} \label{prob: Ursprungsproblem}
	(\nameuncoptprob) \qquad \min_{x \in X \subseteq\R^n} \quad & \objfunc(x,u) \\
	\text{ s.t. } \quad & \consfunc(x,u) \leq 0, \nonumber
\end{align}
depending on an uncertain parameter~$u\in \mathbb{R}^m$. We assume that $f(\cdot,u),g(\cdot,u) : X \to \R$ are at least continuous functions w.r.t$.$ $x$ for some fixed parameter~$u$, which is also called \emph{scenario}, belonging to a \emph{uncertainty set}~$\uncset \subseteq \R^m$. The set $X \subseteq\R^n$ is given by further restrictions on $x$ that do not depend on $u$. For simplicity, we consider only one constraint that depends on the uncertain parameter $u$. However, the following results generalize to multiple constraints by considering their maximum. We assume that there is a special scenario $\bar{u}\in \uncset$ called \emph{nominal scenario}. This could be the average of the scenarios, or the most likely scenario. The nominal problem~$(\mathcal{P}_{\bar{u}})$ is defined as follows: 
\begin{align*}
	(\mathcal{P}_{\bar{u}}) \qquad f^* \coloneqq \min_{x \in X} \quad & f(x,\bar{u})\\
	\text{ s.t. } \quad &  g(x,\bar{u}) \leq 0.
\end{align*}

We call the objective function value of the optimization problem for the nominal scenario above the \emph{nominal objective value} and denote it as $f^*$. Throughout this article we assume that at least the nominal problem has a feasible solution and the nominal objective value $f^*$ is well-defined.

The idea of the \emph{inverse robust optimization problem (IROP)} is to allow a nonnegative deviation~$\epsilon \geq 0$ from the nominal objective value in order to cover the uncertainty set~$\uncset$ as much as possible. We refer to the deviation as the \emph{budget}. The task to cover $\uncset$ as much as possible needs a more precise interpretation. For this, we define a \emph{cover space}~$\coverSpace \subseteq 2^\uncset$ and a \emph{merit function}~$V:\coverSpace \to \R$ which maps every subset of $\uncset$ to a value in $\R$. With this, we obtain an instance of the IROP as follows:

\begin{align} \label{InvRobProblem}
	\inverseRobustProblem \qquad \max_{x \in X, W \in \coverSpace} \quad &  V(W) \\
	\text{ s.t. } \quad & \objfunc(x,u) \leq f^* + \epsilon \qquad \quad \forall u \in W \label{BudgetConstraint}\\
	& \consfunc(x,u) \leq 0 \qquad \qquad \quad  \forall u \in W \label{FeasibilityConstraint}\\
	& \nominalxi \in W \label{NominalScenarioConstraint}.
\end{align}

We call the constraint~\eqref{BudgetConstraint} the \emph{budget constraint} and the constraints~\eqref{FeasibilityConstraint} the \emph{feasibility constraint} of the IROP.

Please note that it is a non-trivial task to define a merit function~$V$ and a cover space $\coverSpace$, since the optimal solution and the tractability depend on it. A bad choice can even lead to an ill-posed problem due to Vitali's theorem (cf.\cite{oxtoby1971MC}). However, this should not be seen as a drawback. These two objects make the definition of a inverse robust optimization problem very general. The merit function can be simply the volume, but can also contain information about the distribution of the uncertain parameter $u$. The cover space can either consist of sets of a concrete shape, e.g. ellipses or boxes, or it can also be a generic set system like a $\sigma$-algebra.

In Section~\ref{sec: cover space choice} we will discuss some concrete choices of the cover space. In this section we are going to show some general statements about the existence and shape of solutions for $\inverseRobustProblem$. One property we want to emphasize here, is that existence of a feasible solution is relatively easy to guarantee. As long as $\{\bar{u}\}\in \coverSpace$, there is a feasible solution, as we assumed that the nominal problem is well-defined. Note that it can be hard to check this for an ordinary robust optimization problem. For the next statements we make some basic assumptions about the cover space $\coverSpace$ and the merit function. Given a compact subset $C\subseteq \uncset$, we denote the set of all compact subsets of $C$ by $\mathcal{K}(C)$.

\begin{assumption}\label{assumption: cover space W}
	We assume that the cover space $\coverSpace$ satisfies the following conditions:
	\begin{enumerate}
		\item For any $W \in \coverSpace$, we know that $\closure{W} \in \coverSpace$.
		\item $\mathcal{K}(C) \cap \coverSpace$ is complete w.r.t. the Hausdorff-metric $d_H$ for any compact subset $C\subseteq \uncset$.
		\item $\{\nominalxi\} \in \coverSpace$.
	\end{enumerate}
\end{assumption}
In the following we let $\tilde{\coverSpace}:= \mathcal{K}(\uncset)\cap \coverSpace$. Note that, if $\uncset$ is itself compact, it suffices to check the second condition in Assumption \ref{assumption: cover space W} for $C=\uncset$.
\begin{assumption}\label{assumption: objective function V}
	Given a cover space $\coverSpace \subseteq 2^\uncset$, we assume that the objective function $V: \coverSpace \to \R$ satisfies the following conditions:
	\begin{enumerate}
		\item $V: \tilde{\coverSpace} \to \R$ is upper semi-continuous w.r.t. the topology induced by the Hausdorff-metric and
		\item $V(W_1) \leq V(W_2)$ for all $W_1, W_2 \in \coverSpace$ with $W_1 \subseteq W_2$.
	\end{enumerate}
\end{assumption}

In the remainder of this section we study how the structure of the parametric problem $(\mathcal{P}_u)$ influences an optimal chosen set $W^* \in \coverSpace$. We start with a theorem that ensures the existence of a solution of $\inverseRobustProblem$.

\begin{theorem}\label{thm: compact set}
	Given a compact uncertainty set~$\uncset \subseteq \R^m$, two continuous functions~$f,g: X \times \uncset \to \R$ w.r.t. $(x,u) \in X \times \uncset$, a compact set~$X \subseteq \R^n$, a cover space $\coverSpace \subseteq 2^\uncset$ and a merit function $V$ which fulfill Assumption~\ref{assumption: cover space W} and Assumption~\ref{assumption: objective function V}. Then there exists a maximizer~$(x^*,W^*) \in X \times \coverSpace$ of $\inverseRobustProblem$, where $W^*$ is a compact set.
\end{theorem}
\begin{proof}
	First we show that if a solution exists, then the corresponding solution set~$W^*$ is a compact set. 
	Let $\closure{W}$ be the closure of a set $W \in \coverSpace$. Because of Assumption~\ref{assumption: cover space W}, we know that $\closure{W} \in \coverSpace$ holds. Due to the continuity of $f,g$ w.r.t. $u$ we can also conclude that for any feasible $(x,W) \in \mathcal{F}$, where
	
	\begin{align*}
		\mathcal{F} \coloneqq \{ (x,W) \in X \times \coverSpace \ : \ & f(x,u) \leq f^* +  \epsilon \ \forall u \in W,\\
		& g(x,u) \leq 0 \ \forall u \in W, \quad \bar{u} \in W \}
	\end{align*}
	holds, also $(x,\closure{W}) \in \mathcal{F}$ is feasible. Since we assumed that $V(W_1) \leq V(W_2)$ for any $W_1,W_2 \in \coverSpace$ with $W_1 \subseteq W_2$, we can reduce the search space of the original optimization problem to the space of closed elements of the cover space $\coverSpace$. As the uncertainty set $\uncset$ was assumed to be compact, we reduce the search space to the space of compact elements of the cover space which is by definition $\tilde{\coverSpace}$.
	
	In a second step, we show that the feasible set
	\begin{align*}
		\tilde{\mathcal{F}} \coloneqq \{ (x,W) \in X \times \tilde{\coverSpace} \ : \ & f(x,u) \leq f^* +  \epsilon \ \forall u \in W,\\
		& g(x,u) \leq 0 \ \forall u \in W, \quad \bar{u} \in W \}
	\end{align*}
	is compact and non-empty. As $\uncset$ is compact, $\mathcal{K}(\uncset)$ is also a compact set itself \footnote{In the set of subsets of $\R^q$ using the topology induced by the Hausdorff-metric, see \cite{Hausdorff1957}).}. Because we assumed that~$\tilde{\coverSpace}$ is complete w.r.t. the Hausdorff-metric $d_H$, we know that it is closed and therefore compact, too. Consequently, the set~$X \times \tilde{\coverSpace}$ is a compact set as the Cartesian product of two compact sets.\\
	
	Next we prove that $\tilde{\mathcal{F}}$ is a closed set. Therefore, we consider a convergent sequence $(x_n,W_n)_{n \in \mathbb{N}} \subseteq \tilde{\mathcal{F}}$ with limit $(x^*,W^*) \in X \times \tilde{\coverSpace}$. We have to show that $(x^*,W^*) \in \tilde{\mathcal{F}}$. We do this by showing that the constraints $\eqref{BudgetConstraint}-\eqref{NominalScenarioConstraint}$ are satisfied.
	
	\begin{itemize}
		\item As $(x_n,W_n) \in \tilde{\mathcal{F}}$, we know that $\nominalxi \in W_n$ holds for all $n \in \mathbb{N}$ and consequently $\nominalxi \in \bigcap_{n \in \mathbb{N}} W_n \subseteq W^*$ as $d_H(W_n,W^*) \to 0$, where $d_H$ is the Hausdorff-metric.
		\item Fix an arbitrary $u^* \in W^*$. As $\lim_{n \to \infty} d_H(W_n,W^*) = 0$, we can find a sequence $(u_n)_{n \in \N}$ with $u_n \in W_n$ for all $n \in \N$ and $u_n \to u^*$. By continuity of $g$ and feasibility of $(x_n, W_n)$ for all $n \in \N$ we get:
		\begin{equation*}
			g(x^*,u^*) = \lim_{n \to \infty} g(x_n,u_n) \leq \lim_{n \to \infty} \max_{u \in W_n} g(x_n,u) \leq 0.
		\end{equation*}
		As $u^* \in W^*$ was chosen arbitrarily this implies $\max_{u \in W^*} g(x^*,u) \leq 0$.
		\item We can argue the same way as for the feasibility constraint $\eqref{FeasibilityConstraint}$ to show: $\max_{u \in W^*} f(x^*,u) \leq f^* + \epsilon$.
	\end{itemize}
	
	This means that all constraints are satisfied and $(x^*,W^*) \in \tilde{\mathcal{F}}$. As the sequence $(x_n,W_n)_{n \in \mathbb{N}}$ was arbitrarily chosen, we showed that $\tilde{\mathcal{F}}$ is closed.\\
	In total we know that the feasible set $\tilde{\mathcal{F}}$ is compact as a closed subset of a compact set.
	
	Because $V$ was assumed to be upper semi-continuous w.r.t. $W$ on $\tilde{\coverSpace}$, we can ensure the existence of a maximizer of $\inverseRobustProblem$. Note that the feasible set $\mathcal{F}$ is non-empty as the choice~$(x^*,\{\bar{u}\})$ is feasible by definition of $f^*$ for all budgets $\epsilon \geq 0$.
\end{proof}

In the statement above we assumed that $\uncset$ is compact. We will now drop this assumption, but demand that the function $V$ is a finite measure on a $\sigma$-algebra.

\begin{theorem}\label{thm: Existence Measure}
	Assume that $\coverSpace$ is a $\sigma$-algebra on $\uncset$ and $V: \coverSpace \rightarrow \mathbb{R}$ is a finite measure. Let $X$ be a compact set, $f,g: X\times \uncset \rightarrow \mathbb{R}$ be continuous functions  and let Assumption \ref{assumption: cover space W} hold. Moreover, assume that there is a sequence of compact sets $C_k \in \coverSpace, k\in \N$, such that $C_{k}\subseteq C_{k+1}$ for $k\in\N$ and $\bigcup_{k\in \N} C_k = \uncset$. Then there exists a maximizer $(x^*,W^*) \in X \times \coverSpace$ of $\inverseRobustProblem$.
\end{theorem}
\begin{proof}
	As in the proof of Theorem~\ref{thm: compact set} we can restrict our consideration to closed sets in $\coverSpace$. Note that by assumption the feasible set of $\inverseRobustProblem$ is non-empty and we consider a finite measure, which fulfills Assumption \ref{assumption: objective function V} $(2)$ by definition and guarantees that the objective is bounded. Thus the supremum~$V^*$ exists and we can find a sequence of feasible elements~$(x_n,W_n)_{n \in \N}$ such that
	\begin{equation}
		\lim_{n\rightarrow\infty} V(W_n) = V^*. \label{eq:Convergence against sup}
	\end{equation}
	As $X$ is assumed to be compact, we can find a subsequence which converges towards an $x^* \in X$. We can assume for the remainder that $\lim_{n\rightarrow\infty} x_n = x^*$.
	As we consider a finite measure we can find for each $\delta > 0$ a $k \in \N$ such that
	\begin{equation}
		V(C_k) \geq V(\uncset) - \delta. \label{eq:deltaApprox}
	\end{equation}
	Now $\mathcal{K}(C_k)\cap \coverSpace$ is, as in the proof above, again a compact set. Which implies that for a fixed $k$ the sequence $(W_n \cap C_k)_{n \in \N}$ has an accumulation point $W^*_k$. W.l.o.g. we assume that this accumulation point is unique. Otherwise, we switch notations to the corresponding subsequence.
	
	As $W_k^*$ is a compact set and $V$ is a finite measure, we conclude using Fatou's Lemma that
	\begin{equation*}
		V(W^*_k) \geq V(\limsup_{n \rightarrow \infty} (W_n \cap C_k)) \geq \limsup_{n \rightarrow\infty} V(W_n \cap C_k).
	\end{equation*}
	Because of Equation \eqref{eq:deltaApprox} we moreover know that $V(W_n \cap C_k) \geq V(W_n) -\delta$ for all $n \in \N$. Together with Equation \eqref{eq:Convergence against sup} we receive
	\begin{equation*}
		V(W^*_k) \geq V^* -\delta.
	\end{equation*}
	It is easy to check that $(x^*,W^*)$ is feasible, where we let $W^* = \bigcup_{k\in\mathbb{N}} W^*_{k}$. As $\coverSpace$ is a $\sigma$-algebra we can guarantee $W^* \in \coverSpace$ and by the continuity of measures we have $V(W^*)= \lim_{n \to \infty} V(W_n) = V^*$ such that $(x^*,W^*)$ is a maximizer of $\inverseRobustProblem$.
\end{proof}

After ensuring the existence of a solution, we can ask which properties of the original problem described by $f,g$ and $\uncset$ induce which structure of $W^*$. One property that we will use later in the discussion of an example problem in Section~\ref{sec: examples} is the inheritance of convexity.

\begin{lemma}\label{lemma: solution convex}
	If a given IROP instance has a maximizer~$(x^*,W^*)$, and $f(x^*,\cdot)$, $g(x^*,\cdot)$ are convex functions w.r.t. $u \in conv(\uncset)$ -- where $conv(\uncset)$ denotes the convex hull of $\uncset$ --, the merit function~$V$ satisfies Assumption~\ref{assumption: objective function V} and the cover space satisfies $\tilde{W}^* := conv(W^*) \cap \uncset \in \coverSpace$, then the decision~$(x^*, \tilde{W}^*)$  is also a maximizer of the problem.
\end{lemma}
\begin{proof}
	Let us denote the optimal solution of the IROP instance as $(x^*,W^*)$. We argue by showing that the choice~$(x^*,\tilde{W}^*)$ satisfies $V(W^*) \leq V(\tilde{W}^*)$ and that this choice is feasible w.r.t. the inverse robust constraints.
	
	By definition we know $W^* \subseteq \tilde{W}^* \subseteq \uncset$ and by Assumption~\ref{assumption: objective function V} that implies $V(W^*) \leq V(\tilde{W}^*)$. In order to prove that $\tilde{W}^*$ is feasible, we choose any arbitrary $u \in \tilde{W}^*$. By the definition of $\tilde{W}^*$ there exist $w_1,w_2 \in W^*, \lambda \in [0,1]$ such that
	\begin{align*}
		f(x^*,u) = f(x^*,\lambda w_1 + (1-\lambda) w_2)
	\end{align*}
	holds. Due to the convexity of $f$ w.r.t. $w \in \uncset$ and the feasibility of $W^*$ we know that 
	\begin{align*}
		f(x^*,\lambda w_1 + (1-\lambda) w_2) &\leq \lambda f(x^*,w_1) + (1-\lambda) f(x^*,w_2)\\
		&\leq \lambda (f^* + \epsilon)+ (1-\lambda) (f^*+\epsilon)\\
		&= f^* + \epsilon
	\end{align*}
	holds as well. Since $u \in \tilde{W}^*$ was chosen arbitrarily we know that
	\begin{align*}
		f(x^*,u) \leq f^* + \epsilon \quad \forall u \in \tilde{W}^*.
	\end{align*}
	
	Analogously we show $g(x^*,u) \leq 0 \ \forall u \in \tilde{W}^*$. Furthermore we know that $\extelem \in W^* \subseteq \tilde{W}^*$ and consequently $\tilde{W}^*$ is feasible and the claim holds.
\end{proof}

Next we will show that the continuity of the describing functions $f,g$ w.r.t$.$ $u$ will imply the (relative) closedness of $W^*$ (w.r.t$.$ $\uncset$).

\begin{lemma}\label{lemma: solution closed}
	If a given IROP instance has a maximizer~$(x^*,W^*)$, and $f(x^*,\cdot)$, $g(x^*,\cdot)$ are continuous functions w.r.t$.$ $u$ and the objective function~$V$ satisfies Assumption~\ref{assumption: objective function V} and the cover space satisfies $\tilde{W}^* := \closure{W^*} \cap \uncset \in \coverSpace$, then the decision $(x^*,\tilde{W}^*)$ -- where $\closure{W^*}$ denotes the closure of $W^*$ -- is also a maximizer of the problem.
\end{lemma}
\begin{proof}
	Let us denote the optimal solution of the inverse robust problem as $(x^*,W^*)$. We will argue by showing that the choice $(x^*,\tilde{W}^*)$ satisfies $V(W^*) \leq V(\tilde{W}^*)$ and that this choice is feasible w.r.t$.$ the inverse robust constraints.
	
	By definition we know $W^* \subseteq \tilde{W}^* \subseteq \uncset$ and by Assumption~\ref{assumption: objective function V} this implies $V(W^*) \leq V(\tilde{W}^*)$. Next we show that $\tilde{W}^*$ is feasible:
	
	Therefore we choose any arbitrary $u \in \tilde{W}^*$. By the definition of $\tilde{W}^*$ there exist a sequence $(w_n)_{n \in \mathbb{N}} \subseteq W^*$ such that
	\begin{align*}
		& \lim_{n \to \infty} w_n = u \text{ and} \\
		& f(x^*,w_n) \leq f^* + \epsilon \ \forall n \in \mathbb{N}
	\end{align*}
	
	holds. Due to the continuity of $f$ w.r.t$.$ $w \in \uncset$ we know that 
	\begin{align*}
		f(x^*,u) &= f(x^*,\lim_{n \to \infty} w_n)\\
		& = \lim_{n \to \infty} f(x^*,w_n)\\
		&\leq f^* + \epsilon
	\end{align*}
	holds as well. Since $u \in \tilde{W}^*$ was chosen arbitrarily we know that
	\begin{align*}
		f(x^*,u) \leq f^* + \epsilon \quad \forall u \in \tilde{W}^*
	\end{align*}
	
	Analogously we show $g(x^*,u) \leq 0 \ \forall u \in \tilde{W}^*$. Furthermore we know that $\extelem \in W^* \subseteq \tilde{W}^*$ and consequently $\tilde{W}^*$ is feasible and the claim holds.
\end{proof}

Last, but not least we will specify conditions for the boundedness of $W^*$:
\begin{lemma}\label{lemma: solution bounded}
	If a given IROP instance has a maximizer~$(x^*,W^*)$ and $h(x^*,\cdot) \coloneqq \max\{f(x^*,\cdot), g(x^*,\cdot)\}$ is a coercive function w.r.t$.$ $u$ or $\uncset$ is bounded, then the set $W^*$ is bounded.
\end{lemma}
\begin{proof}
	For the sake of contradiction, we assume that $W^*$ is unbounded. If $\uncset$ is bounded, this is a contradiction to $W^* \subseteq \uncset$. If $h(x^*,\cdot)$ is coercive and $W^*$ is unbounded, then there exists a sequence $(w_n)_{n \in \mathbb{N}}$ such that $w_n \in W^*$ and $\lim_{n \to\infty} ||w_n|| = \infty$. As $(x^*,W^*)$ is assumed to be a maximizer and therefore is feasible, we conclude that
	\begin{align*}
		h(x^*,w_n) \leq \max\{f^*+\epsilon,0\} < \infty \ \forall n \in \N
	\end{align*}
	holds. This contradicts the coercivity of $h(x^*,\cdot)$ that guarantees for every unbounded sequence $(u_n)_{n \in \mathbb{N}} \subseteq \uncset$
	\begin{align*}
		\lim_{n \to \infty} h(x^*,u_n) = \infty. 
	\end{align*}
	This settles the proof.
\end{proof}

\section{Choice of cover space}\label{sec: cover space choice}
Given an optimization problem as in $\inverseRobustProblem$, we have to specify the cover space $\coverSpace$ to define the problem. This section illustrates some example cover spaces which satisfy Assumption~\ref{assumption: cover space W} such as the whole power set, the Borel-$\sigma$-algebra of the uncertainty set or parameterized families of subsets. These cover spaces can be used together with Theorem~\ref{thm: compact set} to generate a solution of the $\inverseRobustProblem$.

\textbf{The whole power set.} At first we consider the whole power set $\coverSpace = 2^\uncset$ and show that it satisfies Assumption~\ref{assumption: cover space W}. Therefore, we assume that the uncertainty set $\uncset$ is compact. We then know that for an arbitrary $W \in 2^\uncset$ the closure $\closure{W} \subseteq \uncset$ and therefore $\closure{W} \in \coverSpace$. This means that the first condition of Assumption~\ref{assumption: cover space W} holds. The second condition
\begin{align*}
	\mathcal{K}(\uncset) \cap \coverSpace = \mathcal{K}(\uncset) \cap 2^\uncset = \mathcal{K}(\uncset) \text{ is complete}
\end{align*}
holds because $\uncset$ is compact (see \cite{Hausdorff1957}). The last condition holds because $\nominalxi \in \uncset$ and therefore $\{\nominalxi\} \in 2^\uncset = \coverSpace$. Consequently, $\coverSpace = 2^\uncset$ satisfies Assumption~\ref{assumption: cover space W} for any compact uncertainty set $\uncset \subseteq \R^m$. Because the power set is in a sense big enough to contain a solution for $\inverseRobustProblem$, it is not surprising that it satisfies Assumption~\ref{assumption: cover space W}. In the next steps we gradually decrease the size of the cover space.

\textbf{Borel-$\sigma$-algebra.} A more suitable choice, especially if we want to consider measures, is a $\sigma$-algebra. We are interested in the cover space $\coverSpace = \mathcal{B}(\uncset)$ where $\mathcal{B}(\uncset)$ denotes the Borel-$\sigma$-algebra on the closed set $\uncset$.\\
By definition the Borel-$\sigma$-algebra contains all closed subsets of $\uncset$, especially all compact sets and $\{\nominalxi\}$. Therefore, the first and last condition of Assumption~\ref{assumption: cover space W} hold if $\uncset$ is a closed set. As the Borel-$\sigma$-algebra on a close set contains also all compact sets, the completeness condition then also follows. This means that for a closed set $\uncset$ Assumption~\ref{assumption: cover space W} is satisfied.

Also the additional assumptions of Theorem \ref{thm: Existence Measure} on a cover space $\coverSpace$ holds for the Borel-$\sigma$-algebra. As the sequence $(C_k)_{k \in \N}$ of compact sets unit balls around the nominal solution with increasing radius $k\in \N$ can be considered.

\textbf{Sets described by continuous inequality constraints.}
Another step towards a numerically more controllable cover space is done by considering 
\begin{align*}
	& \coverSpace = \{ W(\delta), \delta \in \mathcal{C}(\uncset,\R)\}\\
	\text{with } & W(\delta) = \{u \in \uncset \ : \ \delta(u) \leq 0 \}.
\end{align*}
In this cover space each element is described by a continuous inequality constraints on $\uncset$. Specifying
\begin{align*}
	\delta_{\nominalxi}: \uncset \to \R, u \mapsto \|u-\nominalxi\|,
\end{align*} 
we can guarantee that $W(\delta_{\nominalxi}) = \{\nominalxi\}$ is in $\coverSpace$. Furthermore, the inclusion $\mathcal{K}(\uncset) \subseteq \coverSpace$ holds as for any compact set $A \in \mathcal{K}(\uncset)$ the distance function 
\begin{align*}
	\delta_A : \uncset \to \R, u \mapsto d(u,A)
\end{align*}
is continuous. Because for a compact $A$ the points satisfying $\delta_A(u) \leq 0$ are exactly the points $u \in A$, we can conclude that $A = W(\delta_A) \in \coverSpace$ holds for an arbitrary $A \in \mathcal{K}(\uncset)$.\\ 
Consequently, $\mathcal{W}$ satisfies Assumption~\ref{assumption: cover space W}.

\textbf{Sets described by a family of continuous inequality constraints.}
Last, but not least, we consider cover spaces that are induced by elements of a \emph{design space} $D \subseteq \R^q$. Using so called \emph{design variables} $d \in D$ we focus on the cover space
induced sets
\begin{align*}
	& \coverSpace = \{ W(d), d \in D\},\\
	\text{with } & W(d) = \{ u \in \uncset \ : \ v(u,d) \leq 0 \},
\end{align*}
where $v(\cdot,d): \uncset \to \R$ is a continuous function w.r.t$.$ $u \in \uncset$ for all $d \in D$. Consequently, all sets $W(d)$ are closed for any $d \in D$ such that the first condition of Assumption~\ref{assumption: cover space W} is fulfilled by construction. The other two conditions will not automatically hold and depend on the choice of the function $v$ and the set $D$. As an easy positive example one could think about $v(u,d) = ||u-\nominalxi|| - d$. This way the function $v$ induces the elements $W(d) = B_d(\nominalxi)$ for $d \in D$. The choice $d = 0$ ensures $\{\nominalxi\} \in \coverSpace$. Choosing $D=[0,r]$ for some $r\in \mathbb{R}$ will then ensure that $\coverSpace$ satisfies Assumption \ref{assumption: cover space W}.

However it is possible to construct examples where there exists no solution to $\inverseRobustProblem$. Consider for example $\uncset=[-1,1]$, $\nominalxi=0$, $D=[0,1]$ and
\begin{equation*}
	v(u,d):=\max(u\cdot(1-d),-d-u).
\end{equation*}
We then can set $W(d)=[-d,0]$ for all $d \in [0,1)$, however for $d=1$ obtains $W(1)=[-1,1]$. Given $X=[-1,1]$, the objective function as $f(x,u):=x$ and the constraint as $g(x,u):=0.5-x+u$, and consider the corresponding inverse robust problem. Here, we can define a feasible point $(x,d)$ for every $d \in [0,1)$. If we take the length of the interval $W(d)$ as a merit function, we are interested in the choice $d = 1$. Because $(x,1)$ is always infeasible for any $x \in [-1,1]$, there exists no solution to $\inverseRobustProblem$.

This is not surprising. The choice of a finite dimensional design space $D \subseteq \R^q$ with $q \in \N$ reduces the inverse robust optimization problem to a \emph{general semi-infinite problem (GSIP)} as we can rewrite $\inverseRobustProblem$ in this case as:
\begin{align*} \label{GSIPReformulation}
	\GSIPApproximation \qquad  \max_{x \in \R^n, d \in D} &  V(W(d)) \\
	\text{ s.t. }  \quad & \objfunc(x,u) \leq f^* + \epsilon \quad \ \forall u \in W(d)\\
	& \consfunc(x,u) \leq 0 \qquad \quad \ \ \forall u \in W(d)\\
	&\nominalxi \in W(d)
\end{align*}
For GSIP it is well known that the solution might not exist. For a more detailed discussion we refer to \cite{Stein2003}. A survey of GSIP solution methods is given in \cite{Stein2012}.

A possibility to ensure the existence of a solution and to design discretization methods is to assume the existence of a fixed compact set $Z\subseteq \mathbb{R}^{\tilde{m}}$ and a continuous transformation map $t: D\times Z \rightarrow \mathbb{R}^m$, such that for every $d \in D$ holds
\begin{equation*}
	t(d,Z)=W(d).
\end{equation*}
In this case the GSIP reduces to a standard semi-infinite optimization problem and a solution can be guaranteed by assuming compactness of $X$. This idea is used by the transformation based discretization method introduced in \cite{SchwientekSeidel2020}.

\section{Comparison to other robustness approaches} \label{sec: concept comparison}
As we have pointed out in the introduction, there exist several concepts similar to the inverse robustness. Here we briefly discuss how the \emph{stability radius}, the \emph{resilience radius} and the \emph{radius of robust feasibility} fit in the context of inverse robustness. 

\subsection{Stability radius and resilience radius} \label{subsec: IROPvsSR}
The \emph{stability radius} provides a measure for a fixed solution on how much the uncertain parameter can deviate from a nominal value while still being an (almost) optimal solution. There are many publications regarding the stability radius in the context of (linear) optimization. For an overview, we refer to \cite{Weiss2016}. 

Let $\bar{x} \in \R^n$ denote an optimal solution to a parametrized optimization problem with fixed parameter $\extelem \in \uncset$ of the form
\begin{align*}
	\min_{x \in \bar{X}} \; &  f(x, \extelem),
\end{align*}
where the set of feasible solutions is denoted by~$\bar{X} \subseteq \R^n$. The solution~$\bar{x}$ is called \emph{stable} if there exists an~$\rho> 0$ such that $\bar{x}$ is $\epsilon$-optimal, i.e. $f(\bar{x}, u) \leq f(x, u) + \epsilon$ for all feasible solutions~$x\in \bar{X}$ with an $\epsilon \geq 0$, for all uncertainty scenarios~$u\in B_\rho(\bar{u})$. The \emph{stability radius} is given as the largest such value~$\rho$. Altogether, it can be calculated for a given solution~$\bar{x} \in \bar{X}$ and a budget~$\epsilon \geq 0$ by
\begin{align*}
	\max_{\rho \geq 0} \; & \rho \\
	\text{ s.t. } \; & f(\bar{x}, u) \leq f(x, u) + \epsilon  \quad \forall x\in \bar{X}, \forall u \in B_{\rho}(\bar{u}).
\end{align*}

While the stability radius compares a fixed decision $\bar{x}$ with all other feasible choices $x \in \bar{X}$, the \emph{resilience radius} allows to change the former optimal decision to gain feasibility. For an introduction into this topic we also recommend \cite{Weiss2016}.

Given a budget w.r.t$.$ the objective value, the resilience radius searches the biggest ball centered at a given uncertainty scenario that satisfies feasibility with respect to some original problem. If we denote the optimal solution of a parametrized optimization problem with fixed parameter $\bar{u}$ again by $\bar{x}$, then $\bar{x}$ is called \emph{$B$-feasible} for some budget~$B \in \R$ and some scenario~$u \in \uncset$ if $f_u(\bar{x})$ is lower than $B$.

Then, the \emph{resilience ball} of a $B$-feasible solution $\bar{x}$ around a fixed scenario $\extelem \in \uncset$ is defined as the largest radius $\rho \geq 0$ such that $\bar{x}$ is $B$-feasible for all scenarios in this ball. Finally the resilience radius is the biggest radius of a resilience ball around some~$x \in \bar{X}$ and can be calculated by solving the following optimization problem. 
\begin{align*}
	\max_{x \in \bar{X}, \rho \geq 0} \; & \rho \\
	\text{ s.t. } \; & f(x, u) \leq B  \quad \forall u \in B_{\rho}(\bar{u}).
\end{align*}

To compare these concepts with the concept of inverse robustness, we fix the uncertainty set as $\uncset \coloneqq \R^m$ and define $W(d) \coloneqq B_d(\extelem) \subseteq \uncset, d \in D \coloneqq [0, \infty)$. Furthermore we want to measure $V(W(d)) \coloneqq vol(W(d))$. If we assume that we can describe $\bar{X}$ by finite many inequality constraints, i.e. there exists an finite index set $|I| < \infty$ and continuous functions $g_i : \bar{X} \to \R$ for all $i \in I$ such that $\bar{X} = \{ x \in \R^n \ : \ g_i(x) \leq 0, i \in I\}$ holds, we can define the problem as follows.
\begin{align*}
	\max_{x \in \R^n,d \geq 0} \; & vol(W(d)) \\
	\text{ s.t. } \; & \uncobjfunc(x) \leq f_{\extelem}(\bar{x}) + \epsilon  \quad \forall u \in B_{d}(\bar{u}),\\
	& g_i(x) \leq 0 \quad \forall i \in I, \forall u \in B_d(\extelem).
\end{align*}

This problem can be simplified to the following problem. 
\begin{align*}
	\max_{x \in \bar{X},d \geq 0} \; & d \\
	\text{ s.t. } \; & \uncobjfunc(x) \leq f(\bar{x}, \extelem) + \epsilon  \quad \forall u \in B_{d}(\bar{u}).
\end{align*}

We see that the difference between the stability radius and the inverse robust problem is that the stability radius checks the budget constraint not only for all scenarios $u \in B_\rho(\extelem)$, but also for all feasible $x \in \bar{X}$, while the inverse robust concept allows to choose a new argument~$x \in \bar{X}$ such that the radius is maximized while staying close to the nominal objective value $f^* \coloneqq f(\bar{x}, \extelem)$.

Furthermore, by defining the budget~$B \coloneqq f^*+\epsilon$, we obtain that the resilience radius can be seen as a variant of the IROP, where we are searching an optimal set~$W$ in the set of balls around the nominal scenario~$\bar{u}$.

\subsection{Radius of robust feasibility}
The \emph{radius of robust feasibility} is a measure on the maximal 'size' of an uncertainty set under which one can ensure the feasibility of the given optimization problem. It is discussed for example in the context of convex programs \cite{goberna2016radius}, linear conic programs \cite{goberna2021calculating} and mixed-integer programs \cite{liers2021radius}.

The radius of robust feasibility~$\rho_{RFF}$ is defined as 
\begin{align*}
	\rho_{RFF} & \coloneqq \sup \{ \alpha \geq 0: (\text{PR}_\alpha) \text{ is feasible}\},
	\intertext{where}
	(\text{PR}_\alpha) \quad & \min_{x \in \R^n} c^\top x \\
	\text{ s.t. }\quad  & A x \leq b \qquad \forall (A,b) \in U_\alpha,
\end{align*}

with $U_\alpha \coloneqq (\bar{A},\bar{b}) + \alpha Z$ for nominal values~$\bar{A}\in\R^{m\times n}, \bar{b}\in\R^{m}$ and $Z$ being a compact and convex set. Since we are only interested in the feasibility of $(\text{PR})_\alpha$, we can replace its objective function by $0$. Therefore, given a fixed, convex, compact set~$Z$ we can compute the radius of robust feasibility by solving the following optimization problem:
\begin{align*}
	\rho_{RFF} \quad \coloneqq \sup_{x \in \R^n, \alpha \geq 0} \alpha \\
	\text{ s.t. } \quad & A x\leq b \qquad \forall (A,b) \in U_\alpha,
\end{align*}

with $U_\alpha \coloneqq (\bar{A},\bar{b}) + \alpha Z$. To compare this concept to the concept of inverse robustness, we define $W(d) \coloneqq \extelem + dZ$ as subsets of $\uncset \coloneqq \R^{mn + m}$ characterized by $d \in D \coloneqq [0, \infty)$. Furthermore we use the objective function~ $V(W(d)) \coloneqq vol(W(d))$. Since we do not consider an objective function, we drop the budget constraint. Thus, given a nominal scenario~$\extelem \coloneqq (\bar{A},\bar{b}) \in\uncset$ and a function~$g(x, (A,b)) \coloneqq Ax - b$, we obtain the inverse robust problem
\begin{align*}
	\sup_{x \in \R^n,d \geq 0} \; & vol(W(d))\\
	\text{ s.t. } \;  & g(x, u) \leq 0 \quad \forall u=(A,b) \in \extelem + dZ
\end{align*}

that can be reformulated as
\begin{align*}
	\sup_{x \in \R^n,d \geq 0} \; & d \\
	\text{ s.t. } \; & A x \leq b \qquad \forall (A,b) \in U_\alpha. 
\end{align*}

We see that this way to calculate the radius of robust feasibility can be interpreted as a special inverse robust optimization problem, where we are searching for sets of the form $\extelem + \alpha Z$ and where we are not interested in the budget constraint. The radius of robust feasibility allows us to analyze problems without any pre-defined values such as the given budget $\epsilon \geq 0$ or the nominal solution $f^*$. But, the certain structure of the set~$Z$ is rather restrictive and we do not now how the objective value of a solution~$x$ with a large radius~$\alpha$ deviates from the nominal solution value.

\section{Examples}\label{sec: examples}
After introducing and investigating the concept from a mathematical point of view, we present some further properties using three examples.
\subsection{Dependency on budget}

The first example illustrates that the solution of an inverse robust optimization problem does not depend on the choice of the uncertainty set $\uncset$ in general, but instead on the available budget $\epsilon \geq 0$. Therefore, we focus on the following parametric optimization problem
\begin{align*}
	(\mathcal{P}_u) \qquad \min_{x \in [0,2]} \quad & x + u^2 \\
	\text{ s.t. } \quad & -x+u \leq 0,
\end{align*}
where we consider a parametrized uncertainty set $\uncset(a) = [0,a]$ with $a \geq 1$. Choosing $\nominalxi = 0$ leads to the nominal solution
\begin{align*}
	f^* = 0.
\end{align*}
The corresponding inverse robust optimization problem with $\coverSpace = \{[0,d], d \in [0,a]\}$ and the merit function $V(W) = vol(W)$ has the form
\begin{align*}
	\inverseRobustProblem \qquad \max_{x \in [0,2], d \in [0,a]} \quad & d \\
	\text{ s.t. } \quad & x+u^2 \leq \epsilon \ \forall u \in [0,d]\\
	& -x+u \leq 0 \ \forall u \in [0,d].
\end{align*}
Please note that due to Lemma~\ref{lemma: solution convex} - \ref{lemma: solution bounded} considering the cover space $\mathcal{B}(\uncset)$ would lead to an equivalent problem.\\
The inverse robust optimization problem has the solution $x^* = d^* = -\frac{1}{2} + \sqrt{\frac{1}{4} + \epsilon}$ for $\epsilon \in [0,2]$. This solution is independent of the uncertainty set parameter $a \geq 1$ and thus allows modelling mistakes in the specification of $\uncset$.\\
On the contrary, the corresponding strict robust optimization problem
\begin{align*}
	\qquad \min_{x \in [0,2]} \quad & \max_{u \in [0,a]} x + u^2 \\
	\text{ s.t. } \quad & \max_{u \in [0,a]} -x+u \leq 0
\end{align*}
has the solution $x^*(a) = a$ and $f^*(a) = a^2+a$ for $a \in [1,2]$ and no solution for $a > 2$. This dependence makes it crucial to think about the specification of $\uncset$ beforehand.

\subsection{Extreme scenarios}
In the next example we want to study the effect of extreme scenarios that can occur especially in nonlinear optimization.
We consider for $u \in [0,1]$ the following parameterized optimization problem. 
\begin{align*}
	\min_{x\in [0,1]}&\ x\\
	\text{ s.t. } \enspace&\ x \geq u^{100}
\end{align*}
If we consider the nominal scenario $\overline{u} = 0$, then the nominal objective value $f^*=0$, we receive for chosen budget $\epsilon\geq0$ and the same cover space as before the following inverse robust optimization problem:
\begin{align*}
	\max_{x \in [0,1], d \in [0,1]} &\ d\\
	\text{ s.t. } \enspace &\ x \leq \epsilon\\
	&\ x \geq u^{100} \quad \forall u \in [0,d].
\end{align*}
The optimal solution is given by $x^*=\epsilon$ and $d^*=\epsilon^{\frac{1}{100}}$. On the other hand, choosing an uncertainty set $\uncset=[0,a]$ with $a \in [0,1]$ before solving the classical robust counterpart
\begin{align*}
	\min_{x \in [0,1]} &\ x\\
	\text{ s.t. } \enspace &\ x \geq u^{100} \quad \forall u \in [0,a]
\end{align*}
leads to the optimal solution $x^*=a^{100}$.

A very conservative choice in classical robust optimization would be to choose $a=1$ which would also lead to a high price for robustness and $x^*=1$ as optimal robust solution. In inverse robustness we would first choose a budget $\epsilon$. Lets say $\epsilon=0.1$. The price of robustness we would pay is fixed. The maximal uncertainty set we can cover with this budget has a size of $d^*=0.1^{1/100} \approx 0.977$. This means that we only need pay a price of $0.1$, but cover more than $95\%$ of the area of the original uncertainty set.

Choosing a smaller apriori set with $a=0.5$ leads to a very small price to pay to achieve robustness, $\frac{1}{2^{100}}$. However, if one is ready to pay more for robustness, e.g. $\epsilon = 0.001$ one can cover more than $90\%$ of the area of the original uncertainty set, which is a large part of all scenarios.

The reason for this phenomenon is that $u=1$ is for this problem an extreme scenario. Covering it, has a high price in optimality. In the inverse robust formulation we tend to leave out extreme scenarios and try to find a good solution on the remainder.

The first two examples show two differences to a robust counterpart. First the solution depends directly on the price we are willing to pay to achieve robustness and not on the apriori choice of the uncertainty set. Second the inverse robust optimization will leave out extreme scenarios making it a less conservative approach for robust optimization.

\subsection{A bi-criteria problem} \label{sec: bicriteria problem}
In a final example, we want to demonstrate the flexibility of the new approach. Therefore, we consider a probability measure as the merit function and consider a parameterized bi-criteria optimization problem with an inequality constraint. This constraint is linear with respect to the decision parameter $x \in \R$, but nonlinear in the uncertainty~$u \in \uncset$ such that a solution for a nominal scenario can be easily computed, while the analysis of the behavior with respect to the uncertainty is not trivial. We consider the following bi-criteria optimization problem

\begin{align*}
	\min_{x \in \R} \quad & \{f_1(x,u) \coloneqq -x+u, f_2(x,u) \coloneqq 2x - u\}\\
	\text{ s.t. } & g(x,u) \coloneqq x(u-1) + \exp(u) -1 \leq 0.
\end{align*}

Fixing the nominal scenario~$\nominalxi = 0$, we can compute the Pareto-front $F^*$ as
\begin{align*}
	F^* = \{ t (-1,2)^\top, t \geq 0\}.
\end{align*}

After considering the original problem using a fixed nominal scenario, we now focus on the inverse robust problem. Therefore we allow a generic budget~$\epsilon = (\epsilon_1,\epsilon_2)^\top \in \R_{\geq 0}^2$ and fix a point on the Pareto-front, i.e. $f^* = (-2,4)^\top \in F^*$.

Additionally, we assume that our uncertainty is given by a normal-distributed random variable $u \sim \mathcal{N}(0,1)$. Therefore we let $\coverSpace=\mathcal{B}(\R)$, where $\mathcal{B}(\R)$ denotes the $\sigma$-algebra of Borel-measurable sets of $\R$. We want to maximize the probability of uncertainties we can handle while not loosing more than $\epsilon$ from our solution $f^*$, which leads to:
\begin{align*}
	\sup_{x \in \R, W \in \mathcal{B}(\R)} \, &  \mathbb{P}(u \in W) \\
	\text{ s.t. } \, & f_1(x,u) \leq f^*_1 + \epsilon_1 \qquad \quad \forall u \in W\\
	& f_2(x,u) \leq f^*_2 + \epsilon_2 \qquad \quad \forall u \in W\\
	& g(x,u) \leq 0 \qquad \qquad \qquad \forall u \in W\\
	& 0 \in W
\end{align*}

The statements in Section \ref{sec: model} were all formulated for only one objective. However, it is easy to check that all statements carry over to the case of multiple objectives and can be used to investigate the present example.

As $f_1$ is increasing and $f_2$ is decreasing in $x$ and $0\in W$, we know that depending on the budget $\epsilon$, we can restrict the search space for $x$ to a bounded interval. According to Theorem \ref{thm: Existence Measure} then an optimal solution $(x^*,W^*)$ exists and we can replace the supremum of the last problem by a maximum

As $W \in \mathcal{B}(\R)$ is too large as a search space, we reduce the dimension by searching for intervals $W(d) \coloneqq [d_1,d_2]$ defined by elements of the design space $D \coloneqq \{ d \in \R^2 \ : \ d_1 \leq d_2\}$.
Since $f_1(x,u) = -x+u, f_2(x,u) \coloneqq 2x -u$ are convex w.r.t. $u \in \R$ as linear functions and $g$ is convex w.r.t. $u$ because of $\partial_u^2 g(x,u) = \exp(u) > 0$ for all $u \in \R$, we can use Lemma~\ref{lemma: solution convex}. As the describing functions $f_1,f_2,g$ are continuous w.r.t. $u$ we can use Lemma~\ref{lemma: solution closed} and by Lemma~\ref{lemma: solution bounded} we are looking for a bounded solution set as $h(x,u) = \max\{f_1(x,u),f_2(x,u),g(x,u)\}$ is a coercive function w.r.t. $u$ for any arbitrary $x \in \R$. Consequently the choice of $W(d) = [d_1,d_2], d_1,d_2 \in \R$ to search for a convex, closed, bounded set in $\R$ is appropriate.

We collect this simplification in the following proposition a proof is given in the Appendix.
\begin{proposition}[Reduced problem reformulation]\label{prop: reduced problem formulation}
	The inverse robust example problem can be simplified to the reduced inverse robust example problem given as:
	\begin{align*}
		\reducedExampleProblem \max_{x \in \R,d_1,d_2 \in \R^2} & \, \mathbb{P}(u \leq d_2) - \mathbb{P}(u \leq d_1) \\
		\text{ s.t. } & \,  -x+d_2 \leq -2 + \epsilon_1\\
		& \, 2x-d_1 \leq 4 + \epsilon_2\\
		& \, x(d_2 - 1) + \exp(d_2) - 1\leq 0\\
		& \, d_1 \leq 0\\
		& \, 0 \leq d_2 \leq 1\\
		& \, 0 \leq x
	\end{align*}
	Furthermore, this problem is a convex optimization problem w.r.t. $(x,d)^\top \in X \times D$ and has a solution for all $\epsilon \in \R^2_{\geq 0}$.
\end{proposition}

In Figure~\ref{fig: obj func values} the objective values for different budget $\epsilon_1$ and $\epsilon_2$ are shown. We start with a solution that does not allow any uncertainty, i.e. $P(W(d^*)) = 0$ for $\epsilon_1 = \epsilon_2 = 0$. If we allow to differ from the nominal values~$f_1^*$ or $f_2^*$, we see that we can first gain more robustness by increasing $\epsilon_1$. For each $\epsilon_2$ there is an $\epsilon_1'$ such that for $\epsilon_1 \geq \epsilon_1'$ the solution does not change anymore. A proof of this can be found in Proposition~\ref{prop: increasing budgets} in the appendix. 
For larger $\epsilon_2$ the objective value converges towards $\mathbb{P}(u\leq 1) \approx 0.842$. By the equivalent formulation $\reducedExampleProblem$ it is clear that this is an upper bound for IROP. However for large $k\in \N$ the point
\begin{align*}
	x=&\ -k\left(1-\exp\left(1-\frac{1}{k}\right)\right)\\
	d_1=&\ -k\\
	d_2=&\ 1-\frac{1}{k}
\end{align*}
is feasible for $\reducedExampleProblem$ with the budgets $\epsilon_1=0$ and $\epsilon_2 = -2k(1-\exp(1-\frac{1}{k}))+k$. The objective value of this point converges towards $\mathbb{P}(u\leq 1)$.
\begin{figure}[htbp]
	\centering
	\includegraphics[width=0.6 \textwidth]{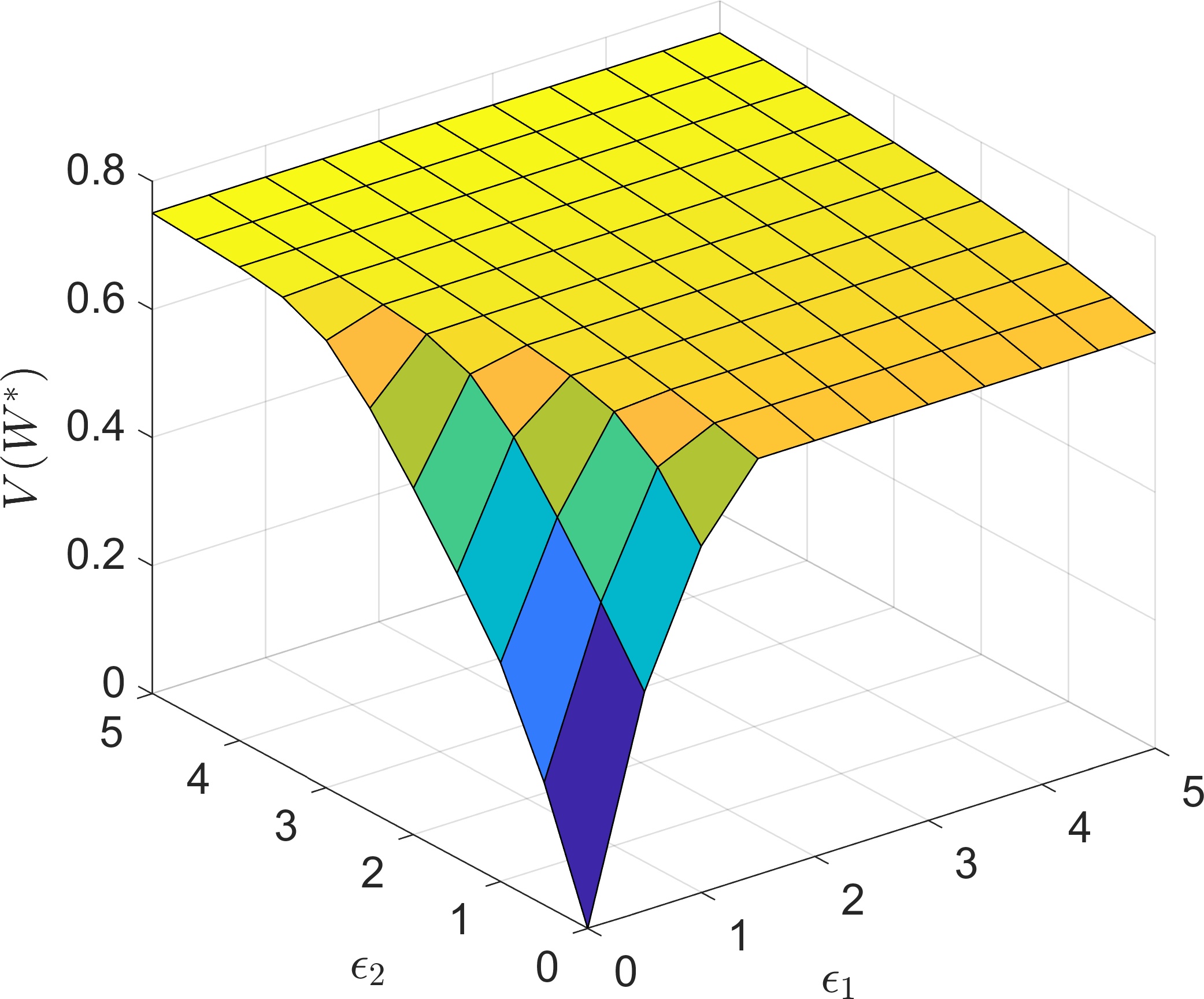}
	\caption{Optimal objective value $V(W^*)$ for different values $\epsilon_1,\epsilon_2 \geq 0$.}\label{fig: obj func values}
\end{figure}

Some of the optimal solution sets~$W^*$ and the robustified decisions~$x^*$ can be seen in Figure~\ref{fig: opt val fix eps2} and Figure~\ref{fig: opt val fix eps1} for different values of $\epsilon$. One could think that the solution sets satisfy an ordering w.r.t. $\subseteq$ if $\epsilon$ increases component-wise. But as on can see this is in general not the case as changes to the decision~$x^*(\epsilon)$ could destroy these inclusions. 

\begin{figure}[htbp]
	\centering
	\includegraphics[width= 0.6 \textwidth]{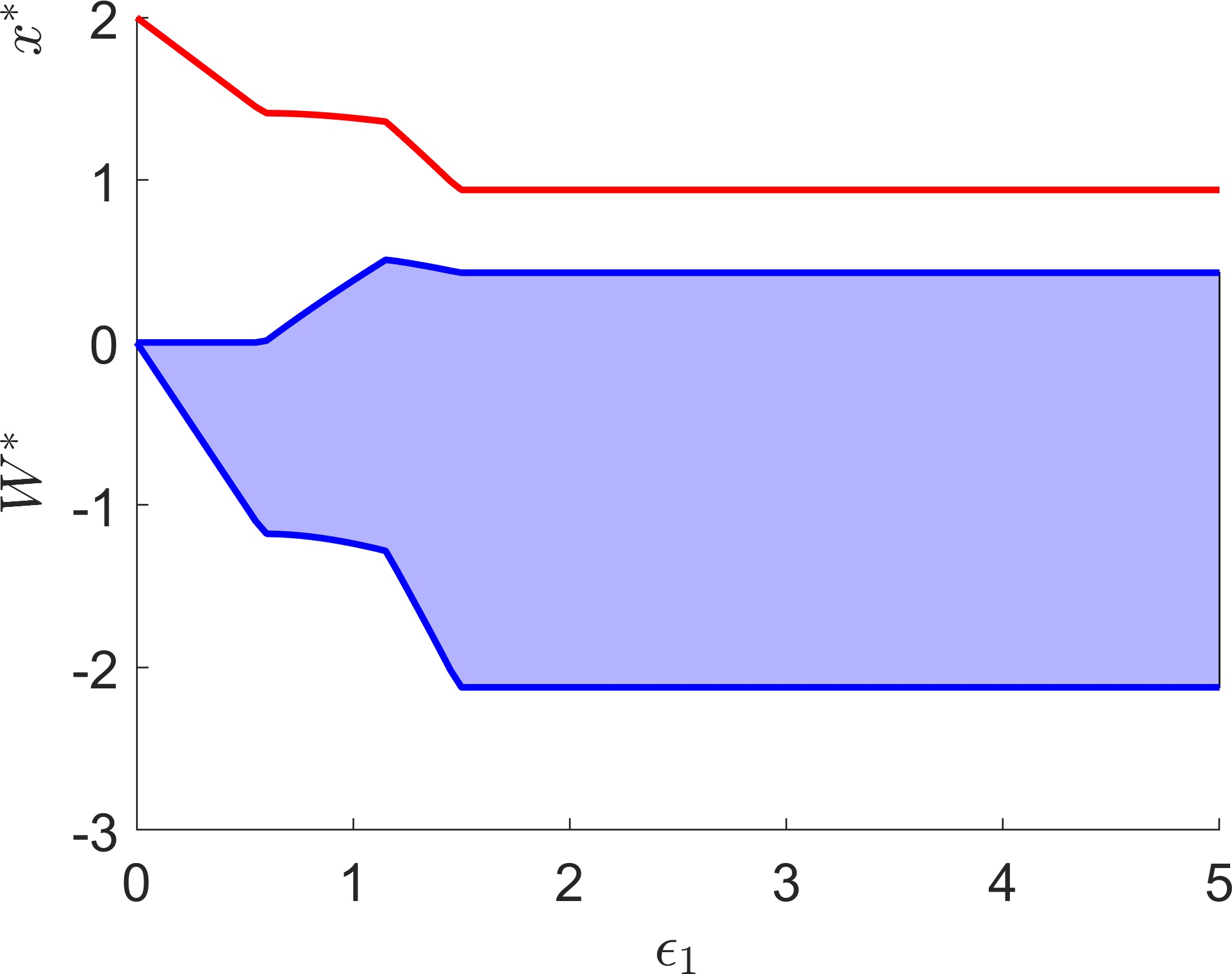}
	\caption{Optimal arguments $x^*$ as red line and $W^*$ as blue area for different values $\epsilon_1$ while fixing $\epsilon_2 \coloneqq 0$.}\label{fig: opt val fix eps2}
\end{figure}
\begin{figure}[htbp]
	\centering
	\includegraphics[width= 0.6 \textwidth]{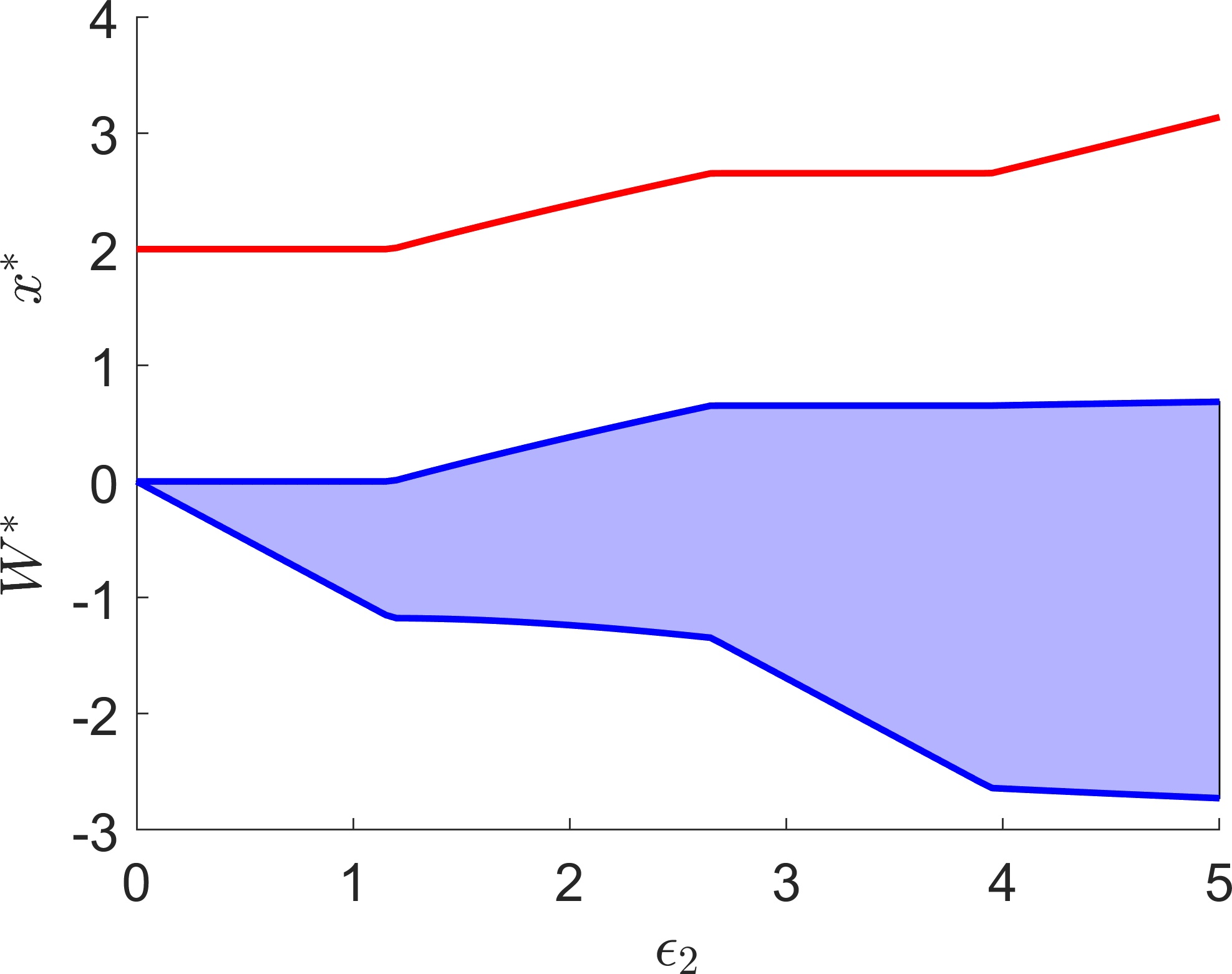}
	\caption{Optimal arguments $x^*$ as red line and $W^*$ as blue area for different for different values $\epsilon_2$ while fixing $\epsilon_1 \coloneqq 0$.}\label{fig: opt val fix eps1}
\end{figure}

\section{Conclusion}\label{sec: conclusion}
Given a parameterized optimization problem, a corresponding nominal scenario, and a budget, one can ask for a solution that is close to optimal with respect to the objective function value of the nominal optimization problem, while being feasible for as many scenarios as possible. 

In this article, we introduced an optimization problem to compute the best coverage of a given uncertainty set. In Section~\ref{sec: model} we introduced the inverse robust optimization problem (IROP) and some structural properties of its solution. In Section~\ref{sec: cover space choice} we discussed different cover spaces that satisfy the assumptions needed for the given structural results of Section~\ref{sec: model}. After comparing IROP with the stability radius, the resilience radius, and the radius of robust feasibility in Section~\ref{sec: concept comparison}, we provided examples in Section~\ref{sec: examples} that demonstrate the flexibility of the concept of inverse robustness. 

\bibliographystyle{plain}
\bibliography{references.bib}

\newpage
\appendix
\section{Properties of Example \ref{sec: bicriteria problem}}

\begin{proof}[Proof of Proposition~\ref{prop: reduced problem formulation}]
	\phantom{.} \ \\
	As discussed in Section~\ref{sec: bicriteria problem} it is enough to consider bounded intervals. Thus, we know that the problem is equivalent to
	\begin{align*}
		\max_{x \in \R, d_1,d_2 \in \R} \, & \mathbb{P}(u \in [d_1,d_2]) \\
		\text{ s.t. } \, & -x+u \leq -2 + \epsilon_1 \qquad \forall u \in [d_1,d_2]\\
		& 2x-u \leq 4 + \epsilon_2 \qquad \forall u \in [d_1,d_2]\\
		& x(u-1)+\exp(u)-1 \leq 0 \qquad \forall u \in [d_1,d_2]\\
		& 0 \in [d_1,d_2].
	\end{align*}
	This problem can be reformulated by computing the maxima within the budget and feasibility constraints
	\begin{align*}
		&\arg\max_{u \in [d_1,d_2]} -x+u = \{d_2\}\\
		&\arg\max_{u \in [d_1,d_2]} 2x-u = \{d_1\}\\
		&\arg\max_{u \in [d_1,d_2]} x(u-1)+\exp(u)-1 = \{d_2\}.
	\end{align*}
	To determine the maximal argument in the feasibility constraint we used the identity $\partial_u g(x,u) = x + \exp(u)$ and that $0 \in [d_1,d_2]$ implies that $g(x,0) = -x \leq 0$ is a necessary condition for a feasible choice of $x$. Therefore $\partial_u g(x,u) > 0$ holds for all feasible choices of $x$ and $u \in \mathbb{R}$. In a last step, we obtain the maximizer~$d_2$ by considering $g(x,u) = \int_{d_1}^u \partial_u g(x,w) dw$.
	
	The last constraint also shows that $d_2 \leq 1$. Otherwise if $d_2>1$ we would violate the feasibility constraint with $x\geq 0$ via
	\begin{equation*}
		x(d_2-1)+ \exp(d_2)-1 > 0.
	\end{equation*}
	
	This means that we receive the following equivalent problem
	
	\begin{align}
		\reducedExampleProblem \max_{x \in \R, d_1,d_2 \in \R^2} \, & \mathbb{P}(u\leq d_2) - \mathbb{P}(u\leq d_1) \\
		\text{ s.t. } \, &  -x+d_2 \leq -2 + \epsilon_1 \label{firstBudgetConstraintReducedProblem}\\
		&\, 2x-d_1 \leq 4 + \epsilon_2 \label{secondBudgetConstraintReducedExampleProblem}\\
		&\, x(d_2 - 1) + \exp(d_2) - 1\leq 0 \label{feasibilityConstraintReducedExampleProblem}\\
		&\, d_1 \leq 0\\
		&\, 0 \leq d_2 \leq 1\\
		&\, 0 \leq x.
	\end{align}
	The objective function is concave in $d_1 \in (-\infty,0]$ and $d_2\in [0,\infty)$. The nonlinear constraint is convex in $x$ and $d_2$, as $x \geq 0$ and $d_2 \in [0,1]$. Since all other constraints are linear w.r.t. $(x,d) \in \R^3$, the reduced problem is a convex optimization problem. The existence of a solution is guaranteed by Theorem \ref{thm: Existence Measure}.
\end{proof}

For the following proposition denote for a given budged $\epsilon$ the optimal solution of the reduced problem $\reducedExampleProblem$ by $x^*(\epsilon),d_1^*(\epsilon)$ and $d_2^*(\epsilon)$

\begin{proposition}[Behavior w.r.t. increasing budgets]\label{prop: increasing budgets}
	Fixing $\epsilon_0 \coloneqq (0,0)^\top$ leads to the solution $x^* = 2, d^* = (0,0)^\top$ and therefore $V(W(d^*(\epsilon_0))) = 0$. For any fixed $\epsilon_1 \geq 0$ we get:
	\begin{itemize}
		\item $\lim_{\epsilon_2 \to \infty} x^*(\epsilon) = \infty$,
		\item $\lim_{\epsilon_2 \to \infty} d^*_1(\epsilon) = -\infty$,
		\item $\lim_{\epsilon_2 \to \infty} d_2^*(\epsilon) = 1$.
	\end{itemize}
	
	For any fixed $\epsilon_2 \geq 0$ and $\epsilon_1 \geq \bar{\epsilon}_1 \coloneqq 3$ the second budget constraint and the feasibility constraint are active. Since the feasibility constraint is independent of~$\epsilon$, it will not change w.r.t. an increasing budget and therefore we obtain
	\begin{itemize}
		\item $\lim_{\epsilon_1 \to \infty} x^*(\epsilon) = x^*(\bar{\epsilon}_1,\epsilon_2)$,
		\item $\lim_{\epsilon_1 \to \infty} d^*_1(\epsilon) = d_1^*(\bar{\epsilon},\epsilon_2)$,
		\item $\lim_{\epsilon_1 \to \infty} d_2^*(\epsilon) = d_2^*(\bar{\epsilon},\epsilon_2)$.
	\end{itemize}
\end{proposition}

\begin{proof}[Proof of Proposition~\ref{prop: increasing budgets}]
	\phantom{.} \ \\
	\begin{itemize}
		\item[i)]  Case $\epsilon = (0,0)^\top$.
		Given the budget~$\epsilon \coloneqq (0,0)^\top$, the reduced inverse robust example problem can be formulated as:
		\begin{align}
			\max_{x \in \R, d_1,d_2 \in \R} \, & \mathbb{P}(u\leq d_2) - \mathbb{P}(u\leq d_1) \nonumber\\
			\text{ s.t. } \, &  -x+d_2 \leq -2 \label{cons: example budget 1}\\
			& 2x-d_1 \leq 4\label{cons: example budget 2}\\
			& x(d_2 - 1) + \exp(d_2) - 1\leq 0\nonumber\\
			& d_1 \leq 0\nonumber\\
			& 0 \leq d_2\nonumber\\
			&\, 0 \leq x. \nonumber
		\end{align}
		
		Considering the budget constraints~\eqref{cons: example budget 1} and \eqref{cons: example budget 2}, we conclude
		\begin{align*}
			x \in \left[ 2+ d_2, 2 + \frac{d_1}{2} \right]
		\end{align*}
		
		Since $d_1 \leq 0, d_2 \geq 0$ has to hold, it follows directly
		\begin{align*}
			x = 2 \; \land  \; d_1 = 0 \; \land \; d_2 = 0.
		\end{align*}
		
		Since this is the only feasible point, it is also the optimal solution of the given problem.
		
		\begin{enumerate}
			\item [ii)] Case $\lim\epsilon_2 \to \infty$.
			We have seen in Section \ref{sec: bicriteria problem} that for $\epsilon_1=0$ and $\epsilon_2$ going to infinity there is a sequence of feasible points such that the objective value converges towards $\mathbb{P}(u\leq 1)$. This means that for the optimal objective value we have
			\begin{equation*}
				\lim_{\epsilon_2 \rightarrow \infty} \mathbb{P}(u\in [d_1^*(\epsilon),d_2^*(\epsilon)]) = \mathbb{P}(u\in (-\infty,1]).
			\end{equation*}
			This is only possible if  
			\begin{align*}
				& \lim_{\epsilon_2 \to \infty} d^*_1(\epsilon) = -\infty\\
				& \lim_{\epsilon_2 \to \infty} d^*_2(\epsilon) = 1.
			\end{align*}
			Considering the feasibility constraint we receive
			\begin{equation*}
				x^*(\epsilon) \geq \frac{\exp(d_2^*(\epsilon))-1}{1-d_2^*(\epsilon)}.
			\end{equation*}
			This shows that we have $\lim_{\epsilon_2 \to \infty} x^*(\epsilon)=\infty$.
			
			\item[iii)] Case $\lim \epsilon_1 \to \infty$.
			
			Let us fix an arbitrary $\epsilon_2 \geq 0$. If we analyze the reduced inverse robust example problem again, we can rewrite its first budget constraint as
			\begin{align*}
				d_2 \leq -2+\epsilon_1+x.
			\end{align*}
			
			As we know that the variable $d_2$ is bounded above by $1$ and we already mentioned that a feasible $x$ has to satisfy $x \geq 0$. Consequently the first budget constraint is fulfilled for all $\epsilon_1 \geq 3$.\\
			Because $\epsilon_1$ just occurs in the first budget constraint of the reduced inverse robust example problem, we know that for $\epsilon_1 \geq 3$ the solution of the problem instance just depends on the choice of $\epsilon_2 \geq 0$ what proves the claim.
		\end{enumerate}
	\end{itemize}
\end{proof}

\end{document}